\documentclass[12pt]{article}

\usepackage{amssymb}
\usepackage{amsthm}
\usepackage{amsfonts}
\usepackage{amsmath}
\usepackage[all]{xypic}
\usepackage{verbatim}

\newtheorem{thm}{Theorem}

\newtheorem{claim}[thm]{Claim}
\newtheorem{prop}[thm]{Proposition}
\newtheorem{conj}[thm]{Conjecture}
\newtheorem{defn}{Definition}

\theoremstyle{remark}

\theoremstyle{definition}

\newcommand{\Z}{\mathbb{Z}}
\newcommand{\Q}{\mathbb{Q}}

\newcommand{\C}{\mathbb{C}}

\newcommand{\seq}[4]{\{#1\}_{#2=#3}^{#4}}

\title{Non-linear Recurrences that Quite Unexpectedly Generate Rational Numbers}
\author{Emilie Hogan}

\begin{document}
\maketitle

\begin{abstract}
Non-linear recurrences which generate integers in a surprising way have been
studied by many people. Typically people study recurrences that are linear in
the highest order term. In this paper I consider what happens when the
recurrence is not linear in the highest order term. In this case we no longer
produce a unique sequence, but we sometimes have surprising results. If the
highest order term is raised to the $m^{th}$ power we expect answers to have
$m^{th}$ roots, but for some specific recurrences it happens that we generate
rational numbers ad infinitum. I will give a general example in the case of a
first order recurrence with $m=2$, and a more specific example that is order 3
with $m=2$ which comes from a generalized Somos recurrence.
\end{abstract}

\section{Introduction}
Many people have studied non-linear recurrences that generate
sequences of integers despite the fact that every iteration of the
recurrence requires division by some previous term in the sequence.
These types of non-linear recurrences generally have the following
form
\begin{align}\label{nonLinRecurDef}
a(n) = L(a(n-1),\ldots, a(n-k))
\end{align}
where $L$ is a Laurent polynomial with integer coefficients, i.e.- $L$ is in
the set $\Z[x_1^{\pm 1}, \ldots, x_{k}^{\pm 1}]$. Well studied examples of this
phenomenon are the Somos sequences, introduced by Michael Somos in 1989
\cite{DGale}, defined by the recurrence
\begin{align*}
s(n)s(n-k)=\sum_{i=1}^{\left\lfloor\frac{k}{2}\right\rfloor}
s(n-i)s(n-k+i)
\end{align*}
with initial conditions $s(m)=1$ for $m\leq k$. For $k=2,3$ the recurrence
generates the infinite sequence $\{1\}_{n=1}^\infty$. More interestingly, for
$k=4,5,6,7$ it is known that these recurrences each generate an infinite
sequence of integers (\cite{FZ}\cite{DGale}). There are, of course, other
examples of this integrality phenomenon, and many are generalizations of the
Somos recurrences (for some examples see \cite{DGale}).

It seems to be the case that all recurrences of the form
(\ref{nonLinRecurDef}) that have been studied have no exponent on
$a(n)$. In this paper I will discuss recurrences of the form
\begin{align}\label{MnonLinRecurDef}
a(n)^m = L(a(n),a(n-1),\ldots, a(n-k))
\end{align}
where $m>1$, and $L$ is still a Laurent polynomial. I will refer to
$m$ many times in this paper, and (unless otherwise stated) this
will refer to the exponent on $a(n)$ in the left-hand side of a
recurrence of the form (\ref{MnonLinRecurDef}).

In general one would imagine that if $a(n)$ is raised to a power $m>1$ then
what is generated is not a sequence at all, since solving an equation of degree
$m$ yields up to $m$ answers. So I need to introduce the concept of a
\emph{recurrence tree}.
\begin{defn}
A \emph{recurrence tree} is a way of storing the values generated by a
recurrence of the form (\ref{MnonLinRecurDef}) with $m>1$. Solving for the
$n+1^{st}$ term, given a specific $n^{th}$ term requires solving an equation of
degree $m$. This yields up to $m$ possibilities for the $n+1^{st}$ term. We can
store these numbers in a complete $m$-ary tree.
\end{defn}
\noindent For example when $m=2$ we get the following structure
 
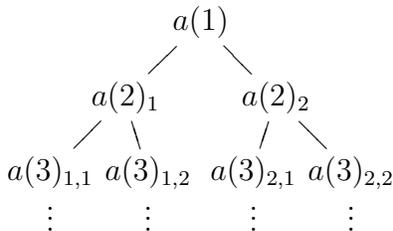
\[
    \xy
    (0,0)*+{a(1)}="a1";
    (-10,-10)*+{a(2)_1}="a21";
    (10,-10)*+{a(2)_2}="a22";
    (-20,-20)*+{a(3)_{1,1}}="a31";
    (-7,-20)*+{a(3)_{1,2}}="a32";
    (7,-20)*+{a(3)_{2,1}}="a33";
    (20,-20)*+{a(3)_{2,2}}="a34";
    (-20,-25)*+{\vdots};
    (-7,-25)*+{\vdots};
    (7,-25)*+{\vdots};
    (20,-25)*+{\vdots};
    {\ar@{-} "a1"; "a21"};
    {\ar@{-} "a1"; "a22"};
    {\ar@{-} "a21"; "a31"};
    {\ar@{-} "a21"; "a32"};
    {\ar@{-} "a22"; "a33"};
    {\ar@{-} "a22"; "a34"};
    \endxy
 \]

Also, since solving an equation of degree $m$ yields answers which can be in
$\C$, we may expect that the numbers generated are not rational. However, in
some cases a recurrence of this form generates rational numbers. When the tree
consists only of numbers in $\Q$ (resp. $\Z$) we will call it a \emph{rational
(resp. integer) recurrence tree}.

One way to come up with recurrences that obviously generate rational
recurrence trees is to take a recurrence that generates integers and
find the ``ratios of ratios" sequence.
\begin{defn}
Given the sequence $\{b(n)\}_{n=1}^\infty$, we call
$\left\{\frac{b(n+1)}{b(n)}\right\}_{n=1}^\infty$ the \emph{sequence
of ratios of $\{b(n)\}$} and
$\left\{\frac{b(n+2)/b(n+1)}{b(n+1)/b(n)}\right\}_{n=1}^\infty$ the
\emph{sequence of ratios of ratios of $\{b(n)\}$}.
\end{defn}
\noindent Obviously, if a sequence $\{b(n)\}_{n=1}^\infty \subset
\Z$ then the sequence of ratios of $\{b(n)\}$ and the sequence of
ratios of ratios of $\{b(n)\}$ are in $\Q$.

Of course, it may not be the case that the recurrence that generates
these ratio sequences has the $m>1$ property, but in the case of the
generalized Somos-4 sequences we can find an alternate recurrence
for the sequence of ratios of ratios that does have this property.
We can then generalize and find new recurrences that do not
obviously generate rational numbers.

\section{Generalized Somos-4 Ratios of Ratios Sequence}\label{GenS4RatRat}

Let $\{s_c(n)\}_{n=1}^\infty$ be a sequence defined by the following
recurrence:
\begin{align}\label{GenSom4}
s_c(n) s_c(n-4)&=c_1s_c(n-1)s_c(n-3)+c_2s_c(n-2)^2
\end{align}
with initial conditions $s_c(i)=1$ for $i\leq4$, where $c=(c_1,c_2) \in \Z^2$ .
This is a special case of the three term Gale-Robinson recurrence
(\cite{FZ}\cite{DGale}) that further specializes to the Somos-4 recurrence when
$c_1=c_2=1$. The first few terms of the sequence are
\[ 1,1,1,1,c_1+c_2,c_1^2+c_1c_2+c_2,c_1^3+2c_1^2c_2+c_1c_2+2c_1c_2^2+c_2^3,\ldots \]

Using cluster algebras and the Caterpillar Lemma, Fomin and Zelevinsky proved
that the recurrence (\ref{GenSom4}) generates a sequence of integers \cite{FZ}.

Now, define sequences $\seq{t_c(n)}{n}{1}{\infty}$, and
$\seq{a_c(n)}{n}{1}{\infty}$ by
\begin{align*}
t_c(n)&=s_c(n+1)/s_c(n)\\
a_c(n)&=t_c(n+1)/t_c(n)
\end{align*}
 then
$\seq{t_c(n)}{n}{1}{\infty}$ is the sequence of ratios of $s_c(n)$
\[\{t_c(n)\}=\left\{1,1,1,c_1+c_2,\frac{c_1^2+c_1c_2+c_2}{c_1+c_2},\frac{c_1^3+2c_1^2c_2+c_1c_2+2c_1c_2^2+c_2^3}{c_1^2+c_1c_2+c_2},\ldots\right\}\]
and $\{a_c(n)\}$ is the sequence of ratios of ratios of $s_c(n)$.
\[\seq{a_c(n)}{n}{1}{\infty}=\left\{1,1,c_1+c_2,\frac{c_1^2+c_1c_2+c_2}{(c_1+c_2)^2},\ldots\right\}\]
In this paper we will be interested in the sequence $\{a_c(n)\}$. By
algebraic manipulation we can easily find a first order quadratic
recurrence for $a_c(n)$.

\begin{claim}\label{acRecurC} The sequence $\{a_c(n)\}_{n=1}^\infty$ is defined by
the recurrence
\begin{align}\label{acRecur}
a_c(n+2)a_c(n+1)^2a_c(n)=c_1a_c(n+1)+c_2
\end{align}
with initial conditions $a_c(1)=a_c(2)=1$.
\end{claim}
\begin{proof}
We will simply manipulate the recurrence equation for $s_c(n)$ to
look like the recurrence equation (\ref{acRecur}).
\begin{align*}
s_c(n+4)s_c(n) &=
 c_1s_c(n+3)s_c(n+1)+c_2s_c(n+2)^2\\
\frac{s_c(n+4)s_c(n)}{s_c(n+2)^{2}} &=
c_1\frac{s_c(n+3)s_c(n+1)}{s_c(n+2)^{2}}+c_2
\end{align*}
Notice that the $s_c$ term on the right side is $a_c(n+1)$. By multiplying and
dividing by the correct terms on the left side we will get the left side of
(\ref{acRecur}).
\begin{align*}
\frac{s_c(n+4)s_c(n)}{s_c(n+2)^{2}}
&=\frac{s_c(n+4)s_c(n)}{s_c(n+2)^2}\frac{s_c(n+2)^2s_c(n+3)^2s_c(n+1)^2}{s_c(n+2)^2s_c(n+3)^2s_c(n+1)^2}\\
&=\frac{s_c(n+4)s_c(n+2)}{s_c(n+3)^2}\frac{s_c(n+3)^2s_c(n+1)^2}{s_c(n+2)^4}
\frac{s_c(n+2)s_c(n)}{s_c(n+1)^2}\\
&=a_c(n+2)a_c(n+1)^2a_c(n)
\end{align*}
Finally we obtain
\[a_c(n+2)a_c(n+1)^2a_c(n)=c_1a_c(n+1)+c_2\]
which is (\ref{acRecur}).
\end{proof}

Unfortunately, this recurrence for the ratios of ratios of $s_c(n)$
does not satisfy $m>1$. The proof of the next claim will help to
create a recurrence with $m=2$.

\begin{claim}\label{acRecur2C} The sequence generated by the recurrence
(\ref{acRecur}) also satisfies the recurrence
\begin{align}\label{acRecur2}
a_c(n+2)a_c(n+1)^2+a_c(n+1)^2a_c(n)=(2c_1+c_2+1)a_c(n+1)-c_1
\end{align}
\end{claim}
\begin{proof}
Showing the converse, that the sequence defined by (\ref{acRecur2})
satisfies (\ref{acRecur}), will prove this claim because of
uniqueness of the sequence. So assume the sequence
$\{a_c(n)\}_{n=0}^\infty$ is defined by (\ref{acRecur2}). Now let us
define a function $T(n)$ by
$$T(n):=a_c(n+1)^2a_c(n)^2-(2c_1+c_2+1)a_c(n+1)a_c(n)+c_1a_c(n+1)+c_1a_c(n)+c_2$$
I claim that it is enough to show $T(n)=0$, for if this is true then
rearranging terms and dividing both sides by $a_c(n+1)^2a_c(n)$ we
get
\begin{align*}
c_1a_c(n+1)+c_2&=(2c_1+c_2+1)a_c(n+1)a_c(n)-c_1a_c(n)-a_c(n+1)^2a_c(n)^2\\
\frac{c_1a_c(n+1)+c_2}{a_c(n)}&=(2c_1+c_2+1)a_c(n+1)-c_1-a_c(n+1)^2a_c(n)\\
\frac{c_1a_c(n+1)+c_2}{a_c(n+1)^2a_c(n)}&=\frac{(2c_1+c_2+1)a_c(n+1)-c_1-a_c(n+1)^2a_c(n)}{a_c(n+1)^2}
\end{align*}
The RHS of the above equality equals $a_c(n+2)$ because we assumed
the sequence $\{a_c(n)\}$ is defined by the recurrence
(\ref{acRecur2}). Therefore we also have that $a_c(n+2)=LHS$, i.e.-
\begin{align*}
a_c(n+2)=\frac{c_1a_c(n+1)+c_2}{a_c(n+1)^2a_c(n)}
\end{align*}
which is recurrence (\ref{acRecur}). So the sequence
$\seq{a_c(n)}{n}{1}{\infty}$, generated by (\ref{acRecur2}), also
satisfies the recurrence (\ref{acRecur}).

All that is left is showing, by induction, that $T(n)=0$ for all
$n$. For $n=1$ we do the following calculation
\begin{align*}
T(1)&=a_c(2)^2a_c(1)^2-(2c_1+c_2+1)a_c(2)a_c(1)+c_1a_c(2)+c_1a_c(1)+c_2\\
 &= 1\cdot1-(2c_1+c_2+1)\cdot 1 \cdot 1 + c_1 \cdot 1 + c_1 \cdot
 1+c_2\\
 &= 1-(2c_1+c_2+1)+2c_1+c_2\\
 &=0
\end{align*}
Now assume that $T(n-1)=0$ for some $n$. We substitute for
$a_c(n+1)$ in $T(n)$ from (\ref{acRecur2}) and simplify to obtain
\begin{align*}
T(n)=&a_c(n-1)^2a_c(n)^2-(2c_1+c_2+1)a_c(n-1)a_c(n)+c_1a_c(n-1)+\\
        &+c_1a_c(n)+c_2\\
 =& T(n-1)=0
\end{align*}
Therefore by induction, $T(n)=0$ for all $n$ and the claim is
proved. This proof used ideas from Guoce Xin's paper \cite{Xin}.
\end{proof}
Coming out of the proof of Claim \ref{acRecur2C} we get that
$T(n)=0$ is a first order recurrence with $m=2$ for $\{a_c(n)\}$ as
we had hoped. One would expect that, since $\{a_c(n)\}$ is by
definition a single sequence, the recurrence tree for $T(n)=0$ would
somehow consist only of this single sequence. Indeed this is the
case as we prove now.
\begin{claim}
The recurrence tree for
\begin{align}\label{SomO1D2}
a_c(n+1)^2a_c(n)^2-&(2c_1+c_2+1)a_c(n+1)a_c(n)+\\
 &+c_1a_c(n+1)+c_1a_c(n)+c_2=0
\end{align}
with $a(1)=1$, produces a single sequence in the sense that at every
level of the tree there is only one value that we haven't yet seen.
\end{claim}
\begin{proof}
Let $X:=a_c(n+1)$ and $Y:=a_c(n)$, then the first order quadratic
recurrence for the generalized Somos-4 sequence is rewritten as
\begin{align}\label{XYRecur}
Y^2 X^2 + (c_1-(2c_1+c_2+1)Y)X + (c_1Y+c_2) &= 0
\end{align}
Given some value $y_o$ for $Y$ there are two possible values for $X$
which satisfy (\ref{XYRecur}). This corresponds to the fact that
given some $a_c(n)$ there are two possible values for $a_c(n+1)$.
Using the quadratic formula, these two values, in terms of $y_o$,
are
\begin{align*}
y^{\pm}&:=\frac{-(c_1-(2c_1+c_2+1)y_o)\pm
\sqrt{(c_1-(2c_1+c_2+1)y_o)^2-4y_o^2(c_1y_o+c_2)}}{2y_o^2}
\end{align*}
Now, since these are values for $a_c(n+1)$ we substitute them back
in for $Y$ in (\ref{XYRecur}), solve for $X$, and get potentially 4
possible values for $a_c(n+2)$ that come from this specific
$a_c(n)=y_o$. However, when we solve the quadratic equation
$$(y^{+})^2 X^2 + (c_1-(2c_1+c_2+1)y^{+})X + (c_1y^{+}+c_2) = 0$$
for $X$ the two solutions we get are $y_o$ and a large expression in
terms of $y_o, c_1, c_2$ (similarly for $y^-$). This means that in
the $i^{th}$ level of the tree, representing all possible values for
$a_c(i)$, from each of the terms in the $i-1^{st}$ level there is at
most one term that we haven't yet seen. Now, lets look at the second
level given the initial condition (the root) $a_c(1)=1$. We solve
the quadratic equation
\begin{align*}
1^2 X^2 + (c_1-(2c_1+c_2+1)\cdot1)X + (c_1\cdot 1+c_2) &= 0\\
X^2 + (-c_1-c_2-1)X + (c_1+c_2) &= 0\\
X = 1 ~\mathrm{or}~ c_1+c_2
\end{align*}
So on the second level we only have one new term. Therefore, on the
third level, and all subsequent levels, we also only have one new
term.
\end{proof}
Since the recurrence tree for (\ref{SomO1D2}) consists only of
numbers from the sequence of ratios of ratios of $\{s_c(n)\}$, it
must be a rational recurrence tree. So we have found an example of a
non-linear recurrence with $m>1$ that generates rational recurrence
tree. However, this example was constructed in such a way that it
had to generate a rational recurrence tree. In the next section I
will generalize this example to get nontrivial sequences generating
rational recurrence trees.

\section{Generalized First Order Quadratic Recurrence
Tree}\label{Ord1QuadTree} The general form of a first order
non-linear recurrence is
\begin{align}\label{O1D2Gen}
\sum_{i=0}^m P_i(a(n))a(n+1)^i =0
\end{align}
where $P_i(Y)$ is a polynomial in $Y$ of some degree $d_i$. For
example, the sequence of ratios of ratios of generalized Somos-4 has
recurrence given by (\ref{O1D2Gen}) where $m=2$ and
\begin{align*}
P_2(Y)&= Y^2\\
P_1(Y)&= c_1-(2c_1+c_2+1)Y\\
P_0(Y)&= c_2+c_1Y
\end{align*}
For the remainder of this section we will assume that $m=2$,
$d_0=d_1=1$, and $P_2(Y)=Y^2$. Let
\begin{align}\label{specP}
P_2(Y)&=Y^2\notag\\
P_1(Y)&=A_1+A_2Y\\
P_0(Y)&=B_1+B_2Y\notag
\end{align}
where $A_1,A_2,B_1,B_2 \in \Z$. Under certain minimal sufficient
conditions a recurrence of this form will generate a rational
recurrence tree.
\begin{prop}
Let $a(1)=1$ in the recurrence (\ref{O1D2Gen}) with coefficient
polynomials given by (\ref{specP}). If
\begin{enumerate}
\item $A_1=B_2$ and
\item solving for $a(2)$ yields rational
numbers,
\end{enumerate}
then the recurrence generates a rational recurrence tree.
\end{prop}
\begin{proof}
First we will show that $A_1=B_2$ implies that for every term $a_1 =
a(n)$ coming from solving
\[a(n)^2 a(n-1)^2 + (A_1+A_2 a(n))a(n-1) + (B_1+A_1a(n)) = 0\]
with $a(n-1)=a_0$, we get only one new $a(n+1)$. In other words,
solving
\[ a(n+1)^2 a_1^2 + (A_1+A_2 a(n+1))a_1 + (B_1+A_1a(n+1)) = 0\]
for $a(n+1)$ yields the solutions $\{a(n+1),a_0\}$. In recurrence
tree form this looks like:
 \[
    \xy
    (0,7)*+{\vdots};
    (0,0)*+{a(n-1)=a_0}="a1";
    (20,-10)*+{\ddots}="a22";
    (-5,-10)*+{a(n)_1=a_1}="a21";
    (15,-20)*+{a(n+1)_{1,1}}="a31";
    (-15,-20)*+{a(n+1)_{1,2}=a_0}="a32";
    (20,-25)*+{\vdots};
    {\ar@{-} "a1"; "a21"};
    {\ar@{-} "a21"; "a31"};
    {\ar@{-} "a21"; "a32"};
    {\ar@{-} "a1"; "a22"};
    \endxy
 \]

Let $a_o$ be a term in the $n-1^{st}$ level. Its children in the
recurrence tree are the solutions of the quadratic equation
\begin{align*}
a_o^2X^2+(A_1+A_2 a_o)X+(B_1+A_1 a_o)=0
\end{align*}
In other words, possibilities for $a(n)$ given that $a(n-1)=a_o$ are
\begin{align*}
a(n)_1&=\frac{-(A_1+A_2
a_o)+\sqrt{(A_1+A_2a_o)^2-4a_o^2(B_1+A_1a_o)}}{2a_o^2}\\
a(n)_2&=\frac{-(A_1+A_2
a_o)-\sqrt{(A_1+A_2a_o)^2-4a_o^2(B_1+A_1a_o)}}{2a_o^2}
\end{align*}
If $a(n)=a(n)_i$, for $i=1,2$, then to get $a(n+1)$ we solve
\begin{align*}
a(n)_i^2a(n+1)^2+(A_1+A_2 a(n)_i)a(n+1)+(B_1+A_1 a(n)_i)=0
\end{align*}
for $a(n+1)$. The goal is to show that $a_o$ is a solution for $a(n+1)$, so it
is enough to show that $a(n)_i^2a_o^2+(A_1+A_2 a(n)_i)a_o+(B_1+A_1 a(n)_i)=0$
for $i=1,2$. This is nothing but algebraic manipulation that can be easily done
using Maple or any other computer algebra system.

The other assumption that we made is that both answers for $a(2)$
are rational. When we solve for level 3 we know that in each case
one answer will be from level 1, so rational. Then the other answer
must be rational because the product of the two answers is the
constant term in the quadratic polynomial, $B_1+A_1a(2)_i$, which is
rational. Likewise, if levels $n-1$ and $n$ are rational then level
$n+1$ will be rational. So by induction we see that all levels are
rational numbers.
\end{proof}

Even with the stipulation that $A_1=B_2$ and $a(2)_1,a(2)_2$ are
both rational, this more general first order recurrence encompasses
more than just the Somos-4 ratios of ratios. For example, let
$A_1=B_2=1,A_2=5,B_1=8$ with the initial condition $a(1)=1$, then
the recurrence is
\begin{align}\label{example}
a(n-1)^2 a(n)^2 + (1+5a(n-1))a(n) + (8+a(n-1)) = 0
\end{align}
The reader may check that $a(2)_1, a(2)_2$ are in fact rational.
There is no $(c_1,c_2)$ such that (\ref{SomO1D2}), the generalized
Somos-4 ratios of ratios recurrence, is the same as (\ref{example}).

Another fact that is worth pointing out is that these recurrences
produce at most two sequences. In Section \ref{GenS4RatRat} we
showed that the first order quadratic recurrence for $a_c(n)$, the
ratios of ratios of the generalized Somos-4 sequence, generates a
unique sequence. This was partly due to the fact that at level 2 one
of the solutions is $a_c(1)$, so we only have one new solution.
However, it is possible to have two new solutions on level 2. In
this case, on level $n>2$ we could have two new solutions, each
coming from one of the $a(2)$. However, we can never get more than
two sequences.

\section{Recurrence Tree for Generalized Somos-4 Sequence}
So far we have looked at first order nonlinear recurrences that
generate trees. What about higher order nonlinear recurrences that
generate trees? As an example we will look at the generalized
Somos-4 recurrence (\ref{GenSom4}). This recurrence is order 4 and
quadratic, but with $m=1$. Since we have a first order quadratic
recurrence for the ratios of ratios of the generalized Somos-4
sequence given by (\ref{SomO1D2}) we can ``unfold" this recurrence
to get an order 3 recurrence that should be satisfied by the
generalized Somos-4 sequence. Since we assumed that $a_c(n) :=
t_c(n+1)/t_c(n)$ and $t_c(n):= s_c(n+1)/s_c(n)$ we can substitute
into (\ref{SomO1D2}). The recurrence we obtain is
\begin{align}\label{Som4Ord3}
s_c(n+3)^2&s_c(n)^2+c_1s_c(n+2)^3s_c(n)+c_2s_c(n+2)^2s_c(n+1)^2\\
-((2c_1&+c_2+1)s_c(n+2)s_c(n+1)s_c(n)-c_1s_c(n+1)^3)s_c(n+3) = 0
\notag
\end{align}

In Section \ref{GenS4RatRat} we were able to show that the
recurrence tree for the ratios of ratios of the Somos-4 recurrence
is just a sequence in disguise because there is only one new term
per level. One might think that the same happens here since
(\ref{Som4Ord3}) is supposed to generate the generalized Somos-4
sequence, however this is not the case. Instead, we get many new
terms per level, and therefore generate many sequences. Also
surprisingly, we even get some non-integer rational numbers in the
tree. As it turns out, at least one of the sequences in the tree has
a simple closed form.
\begin{prop}\label{Som4Ord3Subseq}
One of the sequences generated by the recurrence (\ref{Som4Ord3}) is
$$s_c(1)=s_c(2)=s_c(3)=1, \{(c_1+c_2)^{f(n)}\}_{n=4}^\infty$$
where $f(n)=\left\lfloor\frac{(n-3)^2}{4}\right\rfloor$.
\end{prop}
\begin{proof}
First note that the sequence $\{f(n)\}_{n=4}^\infty$ is
$$0,1,2,4,6,9,12,16,20$$
I claim that $f(2i+1)-f(2i)=i-1$, and $f(2(i+1))-f(2i+1)=i-1$. The value of
$f(2i)$ is
\begin{align*}
\left\lfloor\frac{(2i-3)^2}{4}\right\rfloor&= i^2-3i+2
\end{align*}
and the value of $f(2i+1)$ is
\begin{align*}
\left\lfloor\frac{(2i+1-3)^2}{4}\right\rfloor&= i^2-2i+1
\end{align*}
Their difference is $i-1$ as claimed. Now, $f(2(i+1))$ is
\begin{align*}
\left\lfloor\frac{(2(i+1)-3)^2}{4}\right\rfloor&= i^2-i+0
\end{align*}
The difference $f(2(i+1))-f(2i+1)$ is indeed $i-1$.

Now we will prove this proposition by induction. First, the base
case: we need to show that $s_c(4)=(c_1+c_2)^0, s_c(5)=(c_1+c_2)^1,
s_c(6)=(c_1+c_2)^2$. This is straightforward computation, substitute
in $s_c(1)=s_c(2)=s_c(3)=1$ in (\ref{Som4Ord3}) and solve for
$s_c(4)$. We get two possibilities, 1 and $c_1+c_2$. In this case we
choose $s_c(4)=1$. Now we again substitute $s_c(2)=s_c(3)=s_c(4)=1$
in (\ref{Som4Ord3}) and solve for $s_c(5)$. This is the same
quadratic equation so we get the same 2 solutions, this time we
choose $s_c(5)=c_1+c_2$. To finish off the base case we substitute
$s_c(3)=s_c(4)=1, s_c(5)=c_1+c_2$ in (\ref{Som4Ord3}) and solve for
$s_c(6)$. Our two possible solutions are
$$c_1^2+c_2c_1+c_2~\mathrm{and}~(c_1+c_2)^2$$
so we choose $s_c(6)=(c_1+c_2)^2$. So the base case is true. Now
assume, as the inductive hypothesis, that the proposition is true up
to $n+2$. We have two possibilities
\begin{description}
\item[Case 1:] $n=2i$
\begin{align*}
s_c(2i)=&(c_1+c_2)^k\\
s_c(2i+1)=&(c_1+c_2)^{k+i-1}\\
s_c(2(i+1))=&(c_1+c_2)^{k+2i-2}
\end{align*}
If we substitute this into (\ref{Som4Ord3}) and simplify, we obtain the
quadratic equation
\begin{align*}
s(2i+3)^2(c_1+c_2)^{2k}-((c_1+c_2+1)(c_1+c_2)^{3k+3i-3})s(2i+3)&+\\
+(c_1+c_2)(c_2+c_1)^{4k+6i-6}&=0
\end{align*}
which we can easily solve using the quadratic formula. Our two possibilities
when solving for $s(2i+3)$ are
\[s(2i+3)= \left\{\begin{array}{l} (c_1+c_2)^{k+3i-2}\\
(c_1+c_2)^{k+3i-3}
\end{array}\right.\]

\begin{comment}
{\allowdisplaybreaks
\begin{align*}
&s(2i+3) = \frac{(c_1+c_2+1)(c_1+c_2)^{3k+3i-3}}{2(c_1+c_2)^{2k}}\pm\\
&\pm\frac{
\sqrt{((c_1+c_2+1)(c_1+c_2)^{3k+3i-3})^2-4(c_1+c_2)^{2k}(c_2+c_1)^{4k+6i-5}}}{2(c_1+c_2)^{2k}}\\
&\phantom{s(2i+3}= \frac{(c_1+c_2+1)(c_1+c_2)^{3k+3i-3}}{2(c_1+c_2)^{2k}} \pm\\
&\phantom{s(2i+3=}\pm \frac{\sqrt{(c_1+c_2+1)^2(c_1+c_2)^{6k+6i-6}-4(c_1+c_2)^{6k+6i-5}}}{2(c_1+c_2)^{2k}}\\
&\phantom{s(2i+3}= \frac{(c_1+c_2+1)(c_1+c_2)^{3k+3i-3} \pm
\sqrt{(c_1+c_2)^{6k+6i-6}(c_1+c_2-1)^2}}{2(c_1+c_2)^{2k}}\\
&\phantom{s(2i+3}=\frac{(c_1+c_2)^{3k+3i-3}\left((c_1+c_2+1)\pm
(c_1+c_2-1)\right)}{2(c_1+c_2)^{2k}}\\
&\phantom{s(2i+3}=\frac{1}{2}(c_1+c_2)^{k+3i-3}\left(c_1+c_2+1 \pm (c_1+c_2-1) \right)\\
&\phantom{s(2i+3)}= \left\{\begin{array}{l} \frac{1}{2}(c_1+c_2)^{k+3i-3}(2c_1+2c_2)\\
\frac{1}{2}(c_1+c_2)^{k+3i-3}(2)
\end{array}\right.\\
&\phantom{s(2i+3)}= \left\{\begin{array}{l} (c_1+c_2)^{k+3i-2}\\
(c_1+c_2)^{k+3i-3}
\end{array}\right.
\end{align*}}
\end{comment}

We expected $s_c(2(i+1)+1)=(c_1+c_2)^{k+2i-2+i}=(c_1+c_2)^{k+3i-2}$
and we have it if we choose the ``$+$" in the quadratic formula.
\item[Case 2:] $n=2i+1$
\begin{align*}
s_c(2i+1)=&(c_1+c_2)^k\\
s_c(2(i+1))=&(c_1+c_2)^{k+i-1}\\
s_c(2(i+1)+1)=&(c_1+c_2)^{k+2i-1}
\end{align*}
Again we substitute this in to (\ref{Som4Ord3}) to obtain the following
quadratic equation
\begin{align*}
s(2i&+4)^2(c_1+c_2)^{2k}+\\
&-\left((2c_1+c_2+1)(c_1+c_2)^{3k+3i-2}-c_1(c_1+c_2)^{3k+3i-3}\right)s(2i+4)+\\
&+c_1(c_1+c_2)^{4k+6i-3}+c_2(c_2+c_1)^{4k+6i-4}=0
\end{align*}
which we can again solve using the quadratic equation and obtain
\[ s(2i+4) = \left\{\begin{array}{l} (c_1+c_2)^{k+3i-1}\\
(c_1+c_2)^{k+3i-3}(c_2+c_2c_1+c_1^2)
\end{array}\right. \]
\begin{comment}
I'll skip the beginning steps here
\begin{align*}
s(2i+4) =&
\frac{(c_1+c_2)^{3k+3i-3}(c_2+3c_2c_1+c_2^2+2c_1^2)}{2(c_1+c_2)^{2k}}\pm\\
&\pm\frac{\sqrt{(c_1+c_2)^{6k+6i-6}c_2^2(c_1+c_2-1)^2}}{2(c_1+c_2)^{2k}}\\
=& \frac{(c_1+c_2)^{3k+3i-3}(c_2+3c_2c_1+c_2^2+2c_1^2)}{2(c_1+c_2)^{2k}}\pm\\
&\pm\frac{(c_1+c_2)^{3k+3i-3}c_2(c_1+c_2-1)}{2(c_1+c_2)^{2k}}\\
=& \frac{1}{2}(c_1+c_2)^{k+3i-3}\left((c_2+3c_2c_1+c_2^2+2c_1^2)\pm c_2(c_1+c_2-1)\right)\\
=& \left\{\begin{array}{l} \frac{1}{2}(c_1+c_2)^{k+3i-3}(c_2+3c_2c_1+c_2^2+2c_1^2+c_2(c_1+c_2-1))\\
\frac{1}{2}(c_1+c_2)^{k+3i-3}(c_2+3c_2c_1+c_2^2+2c_1^2-c_2(c_1+c_2-1))
\end{array}\right.\\
=& \left\{\begin{array}{l} \frac{1}{2}(c_1+c_2)^{k+3i-3}(2(c_1+c_2)^2)\\
\frac{1}{2}(c_1+c_2)^{k+3i-3}(2c_2+2c_2c_1+2c_1^2)
\end{array}\right.\\
=& \left\{\begin{array}{l} (c_1+c_2)^{k+3i-1}\\
(c_1+c_2)^{k+3i-3}(c_2+c_2c_1+c_1^2)
\end{array}\right.\\
\end{align*}
\end{comment}
Again, if we choose the $``+"$ in the quadratic equation we get
\[s_c(2(i+2))=(c_1+c_2)^{k+2i-1+i}=(c_1+c_2)^{k+3i-1}\] as expected.
\end{description}
\end{proof}
The fact that we were able to find a nice closed form for one of the integer
sequences in this recurrence tree is very surprising. The closed form for the
generalized Somos-4 sequence is in terms of elliptic theta functions
\cite{Hone}, but by finding a lower order recurrence with higher degree we were
able to find a polynomial sequence.

In looking at the tree that the recurrence (\ref{Som4Ord3})
generates I have noticed that something even more general seems to
be true.
\begin{conj}
Let $T$ be the recurrence tree for (\ref{Som4Ord3}). Every integer,
and the numerator and denominator of every (reduced) non-integer
rational number in $T$, are products of terms in the generalized
Somos-4 sequence.
\end{conj}
\noindent Clearly, proposition \ref{Som4Ord3Subseq} is consistent with this
conjecture since $s_c(5)=c_1+c_2$ in the general Somos-4 recurrence
(\ref{GenSom4}).

\section{Maple code}
This subject could not have been studied without the use of a
computer algebra system, Maple in my case. The Maple code
accompanying this paper can be found on my website
\texttt{http://math.rutgers.edu/$\sim$eahogan/maple/}. I created
programs that calculate the recurrence tree for a given recurrence
of any order and any degree. These programs can be found in the file
\texttt{RecurrenceTree.txt}. I also created programs that generate
all recurrences that have rational recurrence trees to a certain
depth (i.e.- if the tree for a specific recurrence is rational up to
a test depth, the program outputs that recurrence). Those can be
found in file \texttt{GenerateRationalRecurrenceTrees.txt}.

\section{Conclusion}
My study of recurrence trees in this paper may be just the beginning. The first
order recurrences I looked at were limited to the case where $m=2$. As $m$
grows it seems less likely that the recurrence trees generated will be
rational. Though it could be the case that a subtree is rational, or perhaps
just a single sequence. The only higher order recurrence tree I investigated
was that of the generalized Somos-4 recurrence. That specific recurrence is not
completely characterized, but I suspect a generalization, along the lines of
section \ref{Ord1QuadTree}, can be made which may yield behavior like
(\ref{Som4Ord3}).

\section{Acknowledgements}
I would like to thank my advisor, Doron Zeilberger, for bringing my attention
to this problem and helping me along the way. This material is based upon work
supported by the U.S. Department of Homeland Security under Grant Award Number
2007-ST-104-000006.  The views and conclusions contained in this document are
those of the authors and should not be interpreted as necessarily representing
the official policies, either expressed or implied, of the U.S. Department of
Homeland Security.

%\bibliographystyle{elsevier}
%\bibliography{bibliog}

\end{document}